\documentclass[12pt]{article}

\usepackage{amsmath, amsthm, graphicx}
\usepackage{amstext}
\usepackage{amssymb}
\addtocounter{footnote}{-1}

\usepackage{graphicx, pifont, color, epsfig, epstopdf}

\newtheorem{thm}{Theorem}[section] 
\newtheorem{prop}[thm]{Proposition}
\newtheorem{lemma}[thm]{Lemma}
\newtheorem{cor}[thm]{Corollary}

\theoremstyle{definition}

\newtheorem{ex}{Example}[section]

\begin{document}
 \title{\sffamily Fixing numbers for matroids
                     \footnotemark}
\author{
    \textsc{Gary Gordon} \\[0.25em]
    {\small\textit{Department of Mathematics,  Lafayette College
                  }} \\[-0.25em]
    {\small\textit{Easton,PA 18042, USA}} \\[-0.1em]
    {\small\texttt{gordong@lafayette.edu}}\\[.75em]
    \textsc{Jennifer McNulty} \\[0.25em]
    {\small\textit{Department of Mathematical Sciences, University
                   of Montana}} \\[-0.25em]
    {\small\textit{Missoula, MT 59812, USA}} \\[-0.1em]
    {\small\texttt{jenny.mcnulty@umontana.edu}}\\[.75em]
    \textsc{Nancy Ann Neudauer} \\[0.25em]
    {\small\textit{Department of Mathematics and Computer Science,
            Pacific University
                   }} \\[-0.25em]
    {\small\textit{Forest Grove, OR 97116, USA}} \\[-0.1em]
    {\small\texttt{nancy@pacificu.edu}}}

\maketitle
\renewcommand{\thefootnote}{*}

\vspace*{-2em}
\begin{abstract}
\noindent
Motivated by work in graph theory, we define the fixing number  for a matroid.  We give upper and lower bounds for fixing numbers for a general matroid in terms of the size and maximum orbit size (under the action of the matroid automorphism group).  We prove the fixing numbers for the cycle matroid and bicircular matroid associated with 3-connected graphs are identical.  Many of these results have interpretations through permutation groups, and we make this connection explicit.

\end{abstract}
\thispagestyle{empty}
\section{Introduction}\label{s:intro}

The fixing number is an invariant that is associated with breaking symmetry in combinatorial objects.  The {\it
fixing number} was introduced by Harary in 2006 \cite{ErHar}, and, separately, by Boutin the same year (as the {\it determining number}) \cite{Boutin}.  This is the minimum number of vertices that need to be fixed  so that no symmetries remain.

While this graph invariant is relatively recent, this concept has appeared before in the study of permutation groups.  In particular, if a group $G$ of permutations acts faithfully on a set $\Gamma$, then the {\it base size} of $G$ is the smallest number of elements of $\Gamma$  whose pointwise stabilizer is the identity.  When $\Gamma$ is the collection of vertices of a graph, this gives the fixing number of the graph.  An interesting survey of recent (and not-so-recent) results on base size and some other related invariants can be found in {\it Base size, metric dimension and other invariants of groups and graphs}
\cite{BailCam}, which traces the history of this subject to the 19th century.

Our interest is in matroids; specifically, we define the fixing number $fix(M)$ for a matroid $M$.  From the perspective of permutation groups, this amounts to finding the base size of the automorphism group, which always acts faithfully on  the ground set of $M$.  In fact, some of the results we present here have direct analogues in the permutation group literature.

A closely related invariant, the {\it distinguishing number}, can also be extended from graphs to matroids. The distinguishing number of a graph, which was first introduced by Albertson and Collins in 1996 \cite{AlCol}, is the minimum number of colors needed to color the vertices in order to destroy all non-trivial symmetries of the graph. Note that coloring each element of a minimum size fixing set in a matroid $M$ with a different color, then coloring the rest of the elements of the matroid with a new color, we have the distinguishing number bounded above by $fix(M)+1$. We believe the distinguishing number deserves a thorough treatment of its own, but we will not do that here, concentrating solely on the fixing number for the remainder of this paper.

The paper is organized as follows.  In Section~\ref{s:def}, we give some of the basic definitions and results we will need.  In Section~\ref{s:bounds}, we present some basic examples and bounds on the fixing number of a matroid. Section~\ref{s:graphs} is devoted to two matroids associated with a graph: the cycle matroid and the bicircular matroid. Our main result here is that the fixing numbers for these two matroids are the same when $G$ is 3-connected (and large enough).

We believe matroid fixing numbers is a promising area for further research. For instance, it would be valuable to compute fixing numbers for other classes of matroids. It would also be of interest to improve the bounds given in Section~\ref{s:bounds}.

We thank Robert Bailey for useful discussions.

\section{Definitions}\label{s:def}

Throughout, we will let  $M$ be a matroid on the ground set $E$ with circuits $\mathcal{C}$.  Matroids can be defined in a variety of equivalent ways; See \cite{GM} or  \cite{JO2011} for a background on matroids.

Let $\mathcal{C}$ be the collection of circuits of a matroid $M$ on the ground set $E$. A {\it matroid automorphism} $\phi : E \rightarrow E$ is a bijection that preserves the circuits, that is,
$\phi(C) \in \mathcal{C}$ if and only if  $C \in \mathcal{C}$.
Of course, a matroid automorphism  also preserves the independent sets, bases, flats, rank, cocircuits, and all other matroid structure.

 A set of elements $S$ of a matroid is a {\textit{fixing set}} if the only automorphism that fixes $S$ pointwise is the identity.
 The {\textit{fixing number}} of a matroid  {$M$} is defined as the minimum size of a fixing set;
 {$$fix(M) = \min \{ r \:|\: M  \mbox{ has an } r\mbox{-element fixing set\}.}$$

We can also give a group--theoretic definition of the fixing number.  For a matroid $M$ and a subset $A \subseteq E$, define the {\it stabilizer of $A$}, $stab(A)$ to be the set of all automorphisms of $M$ that fix each element of $A$.  That is,
$$stab(A)=\{\sigma \in Aut(M) \mid \sigma(a)=a \mbox{ for all } a \in A\}.$$
Then $stab(A) =\bigcap_{a \in A}stab(a)$ is a subgroup of the automorphism group of $M$. Then  $fix(M) = \min_{A \subseteq E}\{|A| \mid stab(A) \mbox{ is trivial}\}$.

If a permutation group $G$ acts faithfully on a set $\Gamma$, a subset $A \subseteq \Gamma$ is a  {\it base} if $A$ is  a fixing set.  The {\it base size} of $G$ is the minimum size of any base, that is,  the smallest number of points in  $\Gamma$  whose pointwise stabilizer is the identity.  Then if $M$ is a matroid, we immediately see the base size of $Aut(M)$ acting on the ground set $E$ is the fixing number $fix(M)$, i.e., the fixing number $fix(M)$ is the minimum size of any base.  (The term {\it base} is a bit unfortunate here, since a {\it basis} for a matroid need not be related to a base as defined here.)  See Bailey and Cameron
 \cite{BailCam} for details on this approach to the subject.

Let $\mathcal{A}$ be the family of subsets $A$ of $E$ with $stab(A)$ trivial, that is, the collection of all fixing sets.  Then $E \in \mathcal{A}$, so the fixing number is well-defined.  Note that the family $\mathcal{A}$ is an order filter in the boolean lattice of subsets of $E$.  We do not explore the structure of this filter here, but concentrate only on the size of the smallest member of $\mathcal{A}$.  It would be interesting to determine the number of  $A \in \mathcal{A}$ with $stab(A)$ trivial.

Note that $fix(M)=fix(M^*)$ since a matroid and its dual have the same automorphism group. This (obvious) fact has an interesting consequence when we consider the cycle matroid of a planar graph in Section~\ref{s:graphs}. It is also immediate that $fix(M_1\oplus M_2)=fix(M_1)+fix(M_2)$ for the direct sum of matroids $M_1$ and $M_2$ (since $Aut(M_1\oplus M_2)=Aut(M_1)\times Aut(M_2)$). The converse of this is discussed in Example~\ref{E:P6}.

\section{Bounds}\label{s:bounds}
A matroid with trivial automorphism group obviously has fixing number 0. We give an example that has additional properties.

\begin{ex} \label{E:trivauto} Let $M$ be the rank 3 matroid in Fig.~\ref{F:trivauto}. Then $Aut(M)$ is trivial, so $M$ has fixing number 0.
\begin{figure}[h]
\begin{center}
\includegraphics[width=3in]{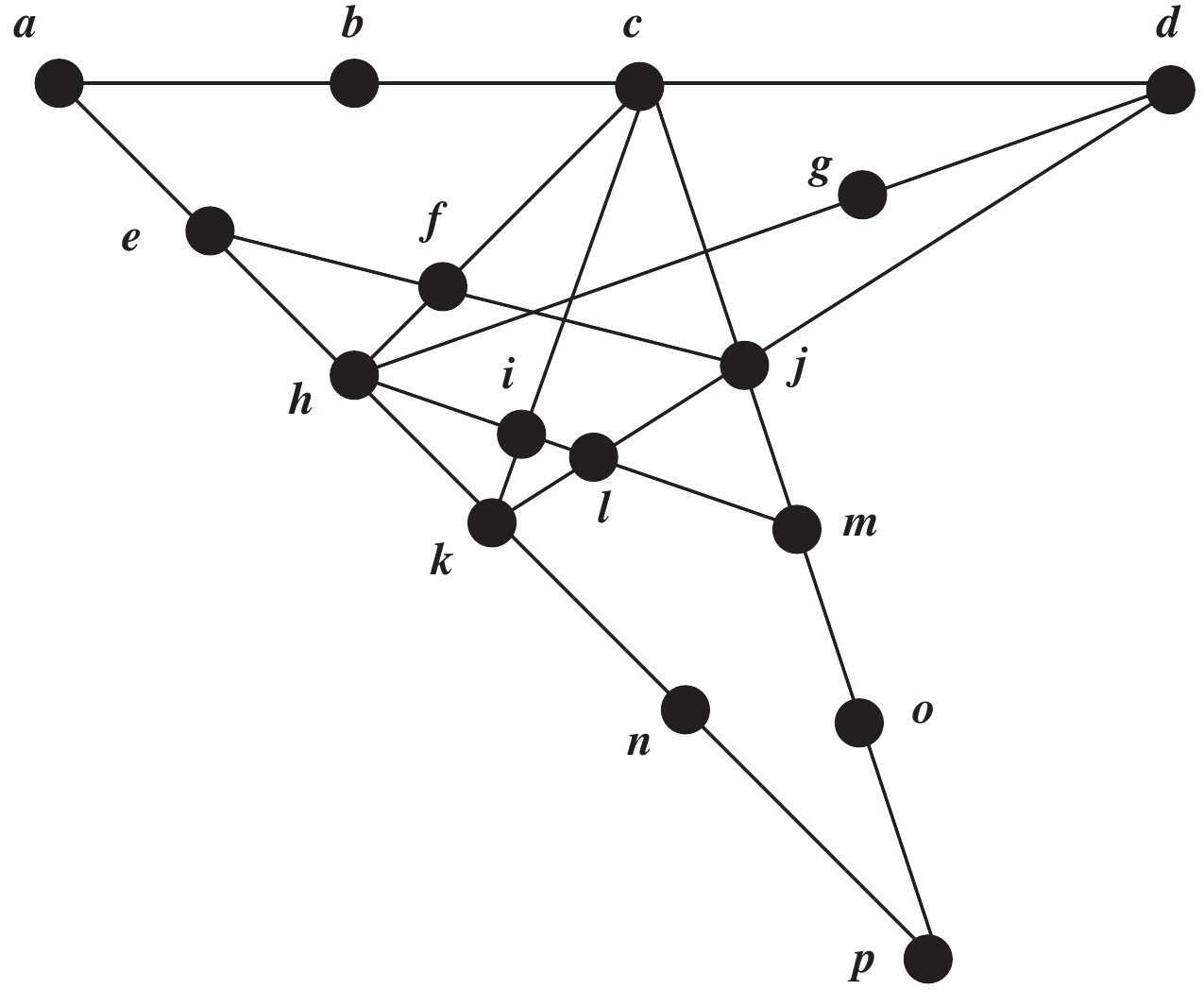}

\caption{A matroid with $M-x \not \cong M-y$ and $M/x \not \cong M/y$ for all pairs $x$ and $y$.}
\label{F:trivauto}
\end{center}
\end{figure}
$M$ has the additional property that, for all points $x$ and $y$, the deletions $M-x$ and $M-y$ are not isomorphic, and similarly for $M/x$ and $M/y$. One would expect that this property (which we interpret as a strong version of lack of symmetry) to predominate in the sense that a random matroid should have this property with probability approaching 1.

\end{ex}

We remark that Example~\ref{E:trivauto} has the following non-matroidal interpretation. Let $\mathcal{P}$ be a finite collection of points in the plane. For each point in $\mathcal{P}$ and for all $k \geq 2$, record the number of $k$-point lines that point is on, and call this the {\it address} of the point. For instance, in Fig.~\ref{F:trivauto}, the  point $h$ is on one 6-point line, one 4-point line, two 3-point lines, and three (trivial) 2-point lines. We  abbreviate the address as $643^22^3$. Then the requirement that all the deletions and contractions are pairwise non-isomorphic ensures that each point has a unique address. The smallest known examples have 13 points  \cite{wagon}.

At the other extreme, the uniform matroid $U_{r,n}$ has symmetry group $S_n$, so has  fixing number $n-1$, as large as possible. The automorphism group of the matroid obviously plays a central role here. We now give upper bounds on the size of the automorphism group of a matroid in terms of the maximum orbit size of a point, extending results of Boutin \cite{Boutin}, who obtains two different  upper bounds on the size of the automorphism group of a graph.  A lower bound on $|Aut(M)|$ in terms of $fix(M)$ appears in \cite{BailCam}.  These results also generalize to matroids:

\begin{thm}\label{T:bounds}  Suppose $M$ is a matroid on a ground set of size $n$, with fixing number $fix(M)=k$ and maximum orbit size $s$.
\begin{enumerate}
\item    $|Aut(M)| \leq (n)_k.$
\item   $|Aut(M)| \leq s^k.$
\item $2^k \leq |Aut(M)|$.
\end{enumerate}
\end{thm}

\begin{proof}

First note that if $B$ is a fixing set with $|B|=k$, and two automorphisms $\phi$ and $\psi$ agree on $B$, then $\phi=\psi$. Thus, the number of distinct automorphisms is bounded above by the number of possible images of the set $B$. Parts 1 and 2 now follow.

For part 3, we follow \cite{BailCam}.  Let $B=\{b_1, b_2, \dots, b_k\}$ be a fixing set of minimum size.  Let $G_i=stab(\{b_1, b_2, \dots, b_i\})$ be the subgroup of $Aut(M)$ that fixes the first $i$ elements of $B$, where $G_0=Aut(M)$.  Then the sequence of subgroups $$Aut(M)=G_0 \geq G_1 \geq G_2 \geq \cdots \geq G_k = \{e\}$$
is strictly decreasing at each step, i.e., $G_{i+1}$ is a proper subgroup of $G_i$.  Thus $|G_i|\geq 2|G_{i+1}|$.  The lower bound on $Aut(M)$ now follows.

\end{proof}

Part 1 of Theorem~\ref{T:bounds} is sharp for uniform matroids, and parts 2 and 3 are sharp for matroids with trivial automorphism group. In general, it would be interesting to study which classes of matroids satisfy these bounds with equality.

  As an example where the bounds are not sharp, the Fano plane $F_7$ has automorphism group isomorphic to the simple group $PSL(2,7)$, with $168=7 \times 6 \times 4$ elements.  The reader can check $fix(F_7)=3$, so the bound in part 1 of Theorem~\ref{T:bounds} gives $|Aut(F_7)|\leq 7\times 6 \times 5=210.$  For part 2 of the theorem, note that this group is transitive on points (in fact, it is transitive on pairs), so there is only one orbit, with 7 points.  Thus we get $|Aut(F_7)|\leq 7^3=343.$   For $F_7$, part 3 gives a (rather poor) lower bound of $2^3=8$ for $|Aut(F_7)|$.

\begin{ex} \label{E:dc} 
The fixing number is not well-behaved under the matroid operations of deletion and contraction. For example, recall that $fix(M)=0$ for the matroid $M$ from Fig.~\ref{F:trivauto}. Let $M'$ be the matroid obtained from $M$ by adding a point {\it freely} without increasing rank (this is the matroid operation of {\it free extension}). Then $fix(M')=0$, but $fix(M'/x)$ is a rank 2 matroid with fixing number 16. So contracting a point (and, dually, deleting a point in $(M')^*$) can increase the fixing number rather dramatically.
\end{ex}

 It is clear that modifying this the matroid $M'$ in Example~\ref{E:dc}  can produce matroids whose fixing numbers increase by arbitrary amounts under deletion and/or contraction of a single point. Note that the automorphism groups for the minors $M-x$ or $M/x$ may be much larger than $Aut(M)$.

\subsection{Clones} There is a direct connection between the fixing number and  {\it clones} in matroids. Let $x, y \in E$ be elements of the ground set of a matroid $M$. If the map that interchanges $x$ and $y$ and fixes everything else is an automorphism of the matroid, we say $x$ and $y$ are {\it clones}. This induces an equivalence relations on $E$, partitioning $E$ into {\it clonal} classes.

In finding a minimum size fixing set for a matroid, it is clear that we must fix (at least) all but one member of each clonal class. The next result is immediate.

\begin{prop}\label{P:clone} Let $M$ be a matroid on a ground set of size $n$ with $m$ clonal classes. Then 
$$fix(M)\geq n-m.$$

\end{prop}

This lower bound is exact for uniform matroids (where there is one clonal class) and matroids with trivial automorphism group (where each element is in its own class). It will also be useful for our treatment of transversal matroids.

It is straightforward to show that two elements are clones precisely when they are in the same cyclic flats of the matroid. This observation is useful for finding clones. 

\begin{ex}\label{E:vamos}
We compute the fixing number of the V\'{a}mos cube $V_8$ (see Fig.~\ref{F:vamos}). This matroid is not representable over any field, and is one of the minimal such examples.

\begin{figure}[h]
\begin{center}
\includegraphics[width=2in]{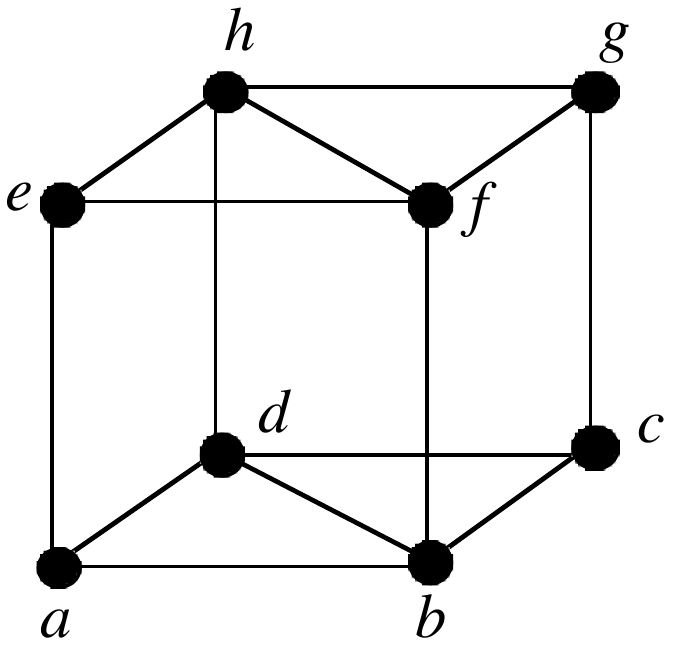}

\caption{V\'{a}mos matroid $V_8$.}
\label{F:vamos}
\end{center}
\end{figure}

There are five 4-point circuits in $V_8 : \{a,b,e,f\}, \{b,c,f,g\}, \{c,d,g,h\}, \{a,d,e,h\}$ and $\{b,d,f,h\}$. These five sets are precisely the non-trivial cyclic flats of $V_8$. Then there are four clonal classes: $\{a,e\}, \{b,f\}, \{c,g\}$ and $\{d,h\}$ (this is direct from the cyclic flat characterization). Then Prop.~\ref{P:clone} gives $fix(V_8)\geq 4$. It is easy to see this is best possible: $fix(V_8)=4$.

\end{ex}

For the V\'{a}mos matroid, the automorphism group is a non-abelian group of order 64. In addition to the $(\mathbb{Z}_2)^4$ action induced by swaps of clones, there is a dihedral $D_2$ action, so we can write $Aut(V_8)\cong (\mathbb{Z}_2)^4 \rtimes D_2$. We list the order of the subgroups in a sequence of stabilizers of minimum length, demonstrating $fix(V_8)=4$. 

\begin{center}
\begin{tabular}{c||c|c|c|c|c} 
Group & $Aut(V_8)$ & $stab(\{a\})$ & $stab(\{a,b\})$ & $stab(\{a,b,c\})$ & $stab(\{a,b,c,d\})$ \\ \hline
Order & 64 & 16 & 4 & 2 & 1
\end{tabular}
\end{center}

\subsection{Transversal matroids}\label{s:trans}
Transversal matroids are an important class of matroids; see  \cite{Brualdi-White} as a reference. We define them as follows:  Let  $G$ be a bipartite graph with vertex bipartition $V=X\cup Y$.  The {\it transversal matroid} $M_G$ is defined on the ground set $X$. A subset $I \subseteq X$ will be independent in $M_G$ precisely when it can be matched into $Y$ in the bipartite graph. It is a standard exercise to prove that the sets that can be matched satisfy the independent set properties that characterize matroids. We will also assume $|Y|=r(M)$, the rank of the matroid (it is always possible to find a bipartite graph with this property to represent a transversal matroid). For $x \in X$, we let $R(x)\subseteq Y$ be the elements of $Y$ joined to $x$ in $G$.

Alternatively, for $y \in Y$, let $R^{-1}(y)\subseteq X$ be the elements of $X$ joined to $y$ in the bipartite graph $G$. Then the independent sets are the systems of distinct representatives ({\it partial transversals})  associated with the family of subsets $\{R^{-1}(y)\}_{y \in Y}$. Different bipartite graphs may give the same matroid. A bipartite graph $G$ is a {\it maximal presentation} for the transversal matroid $M_G$ if adding any edges to $G$ changes the matroid.

Let $M_G$ be a transversal matroid of rank $r.$ Then a maximal presentation can be used to find  the clonal classes of $M_G$. 

\begin{lemma}\label{P:transclone}
Let $M_G$ be a transversal matroid with $G$ a maximal presentation. Then $x_1$ and $x_2$ are clones if and only if $R(x_1)=R(x_2)$.
\end{lemma} 
\begin{proof}
Suppose $x_1$ and $x_2$ are clones. Then replacing $R(x_1)$ and $R(x_2)$ by $R(x_1)\cup R(x_2)$ does not change the matroid (specifically, it does not change the cyclic flats that contain $x_1$ or $x_2$). Since the presentation is assumed to be maximal, we must have $R(x_1)=R(x_2)$.

For the converse, it is clear that if $R(x_1)=R(x_2)$, then there is a {\it graph} automorphism of $G$ that swaps $x_1$ and $x_2$, fixing everything else. This bipartite graph automorphism induces a matroid automorphism which also swaps these two elements, rising everything else.
\end{proof}

Geometrically, one can use the bipartite graph $G$ to construct an affine diagram for a transversal matroid. Details are given in \cite{GM}. In this context, the points in a clonal class are placed freely on the face of a simplex. 

\begin{ex}\label{E:P6}
Consider the matroid $P_6$, shown in Fig.~\ref{F:P6}. ($P_6$ is one of the excluded minors for representability over $GF(4)$.)

\begin{figure}[h]
\begin{center}
\includegraphics[width=1.75in]{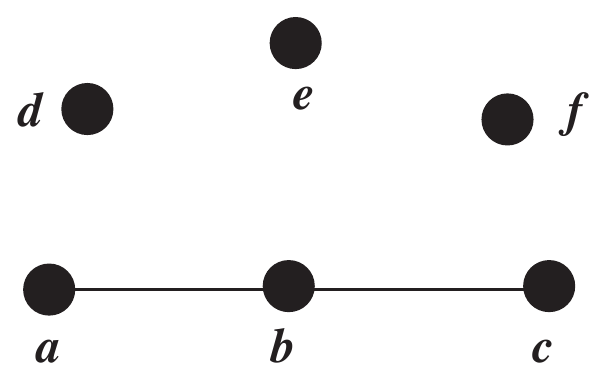}
\caption{The matroid $P_6$.}
\label{F:P6}
\end{center}
\end{figure}
This matroid is transversal, with $X=\{a,b,c,d,e,f\}$ and $Y=\{1,2,3\}$, where  $R^{-1}(1)=R^{-1}(2)=\{a,b,c,d,e,f\},$ and $R^{-1}(3)=\{d,e,f\}$ is a maximal presentation. Then $R(a)=R(b)=R(c)=\{1,2\}$, and $R(d)=R(e)=R(f)=\{1,2,3\}$. This gives us two clone classes: $\{a,b,c\}$ and $\{d,e,f\}$. Hence, $fix(P_6) \geq 4$.  It is easy to check this is best possible, so $fix(P_6)=4$.

Note that, in this case, the clone classes determine the matroid automorphism group: $Aut(P_6)\cong S_3 \times S_3$, the direct product of two symmetric groups. Thus, it is possible for the automorphism group of a matroid to be a direct product when the matroid is not a direct sum.
\end{ex}

Applying Prop.~\ref{P:clone} gives us a lower bound on the fixing number of a transversal matroid.

\begin{cor}\label{C:transclone}
Let $M_G$ be a transversal matroid with maximal presentation $G$. Let $X=B_1 \cup B_2 \cup \cdots \cup B_k$ be a partition of $X$ with the property $R(x_1)=R(x_2)$ iff $x_1$ and $x_2$ are in the same block $B_i$ of the partition. Then 
$$fix(M_G)\geq |X|-k.$$
\end{cor}

Cor.~\ref{C:transclone} may give very poor lower bounds for fixing numbers of transversal matroids. Consider the bipartite incidence graph $G_{n,k}$, where $Y=[n]$ and $X$ is the collection of all $k$-subsets of $[n]$. (When $k=2$, this is the bicircular matroid associated with the complete graph $K_n$ -- see Sec.~\ref{s:graphs}.)  The transversal matroid $M_{n,k}$ is simply the collection of partial transversals of the family of $k$-subsets of $n$ (which form a well-studied uniform hypergraph).

\begin{prop}\label{P:hyper}
Let $M_{n,k}$ be the transversal matroid associated with the family of all $k$-subsets of $[n]$. Then $M_{n,k}$ is a uniform matroid if and only if $k=1$ or $n-k \leq 2$.
\end{prop}
\begin{proof}
If $k=1$, then $M_{n,1}=U_{n,n}$ is a boolean algebra. If $k=n$, there is only one element in the matroid.  If $k=n-1$, then $M_{n,n-1}$ is also a boolean algebra. When $k=n-2$, it is straightforward to check every collection of $n$ of the $k$-subsets has a transversal, so every subset of size $n$ in $M_{n,k}$ is a basis. This gives $M_{n,n-2}\cong U_{n,n(n-1)/2}$. 

When $1<k<n-2$, one can construct a collection of $n$ of the $k$-subsets that do not have a transversal. Thus, there are collections of $n$ elements of the matroid $M_{n,k}$ which do not  form a basis in the matroid, so $M_{n,k}$ is not a uniform matroid. We omit the details.
\end{proof}

When $n-k>2$ (and $k>1$), the transversal matroid $M_{n,k}$ has trivial clone classes, i.e., each class has only one element. Thus, the lower bound given in Cor.~\ref{C:transclone} is trivial. Determining the fixing numbers for these matroids is somewhat technical; see \cite{CGGMP}  and \cite{Maund} for bounds, and Theorems 2.22 and 2.25 of \cite{BailCam} for a summary of what is known for these matroids (where they are considered via the action of $S_n$ on $k$-subsets of $[n]$).

We point out the similarity between Proposition~\ref{P:hyper} and Theorem 3.8 of \cite{LRM}, which describes precisely which uniform matroids are bicircular. Bicircular matroids are transversal; they are considered here in Sec.~\ref{s:graphs}.

When $M$ is a simple binary matroid, we can bound $fix(M)$ in terms of the rank.

\begin{thm}\label{T:bin} If $M$ is a simple binary matroid of rank $r$, then $fix(M) \leq r$.
\end{thm}

\begin{proof}
Let $M$ be a simple binary matroid of rank $r$ and let $B$ be a basis of $M$. We claim $B$ is a fixing set. For all $e \in E-B$, there is a unique circuit $C$ such that $e \in C \subseteq B\cup e$ (this circuit is the {\it fundamental} (also called {\it basic}) circuit determined by $B$ and $e$). Let $\phi$ be an automorphism of $M$ that preserves $B$ pointwise. Then $\phi$ preserves the elements of $C-e$ pointwise and, since $\phi$ preserves circuits, it follows that $\phi(e)=e$.

\end{proof}

Combining this result with part 1 of Theorem~\ref{T:bounds} gives the following corollary.

\begin{cor} If $M$ is a  simple binary matroid of size $n$
and rank $r$, then  $|Aut(M)| \leq (n)_r$.
\end{cor}



The Fano plane shows that the bounds in Theorem~\ref{T:bin} are sharp.  In fact, it is straightforward to show the projective geometry $PG(r-1,2)$ has the fixing number $r$.

\section{Cycle  and  bicircular matroids}\label{s:graphs}

Two matroids  commonly defined on the edges of a graph are the  cycle matroid and
the bicircular matroid.  While
the fixing number of a graph is concerned with fixing the vertices of the graph, 
our goal with these two matroids is fixing edges.  
What is the connection between these two automorphism groups? 


For a graph, $G$,  the {\it cycle matroid} of $G$, denoted $M(G)$, is 
defined on the edge set of the graph where the cycles of the graph are the circuits of this matroid.
The {\it
bicircular matroid} of $G$ is the matroid  $B(G)$ defined on the same edge set
whose circuits are the subgraphs which are the  {\it bicycles} of $G$, i.e., subdivisions of one of the following
graphs:  
two loops incident to the same vertex,
 two loops joined by an edge, or
 three edges joining the same pair of vertices.
 We note that bicircular matroids are transversal.

\begin{figure}[h]

\centerline{\psfig{figure=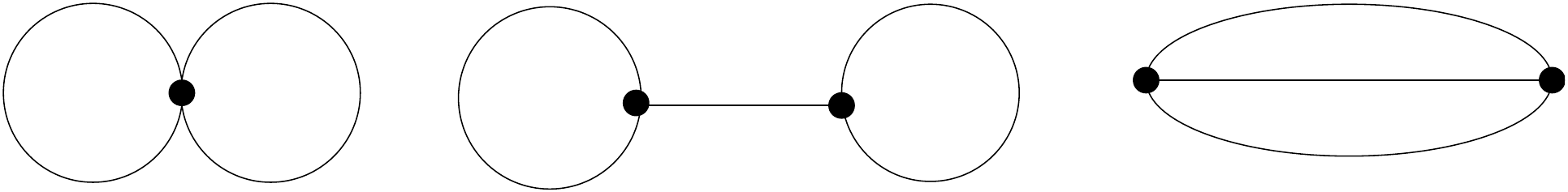,height=0.5in}}
\caption{Bicycles of $G$ \label{fig:bicycles}}
\end{figure}




We approach the automorphism groups and fixing numbers via the
cocircuits of each matroid.
We assume $G$ is a connected graph. For the cycle matroid $M(G)$, a subset $C^*$ of  edges is a cocircuit if 
it forms  a minimal cutset of the graph.
For the bicircular matroid $B(G)$, a subset $C^*$ of
the edges of $G$ is a  cocircuit  if the spanning
subgraph $E- C^*$ has exactly one tree component $T$ and each
edge in $C^*$ is incident with at least one vertex of $T$.  In other
words, $C^*$ is a minimal set of edges such that $E - C^*$ has
exactly one tree component.

\begin{ex}\label{E:autoex1} To see that the cycle and bicircular matroids can have different fixing numbers, consider the theta graph $G$ in Figure~\ref{f:graph}.  Here $fix(M(G))=4$, but $B(G)\cong U_{6,7}$, so $fix(B(G))=6$.  Note that the three automorphism groups are distinct: $Aut(G) \cong D_2$,  
$Aut(M(G)) \cong (S_3 \times S_3) \times \mathbb{Z}_2$, and $Aut(B(G)) \cong S_7$.

\begin{figure}[h]
\begin{center}
\includegraphics[width=1.6in]{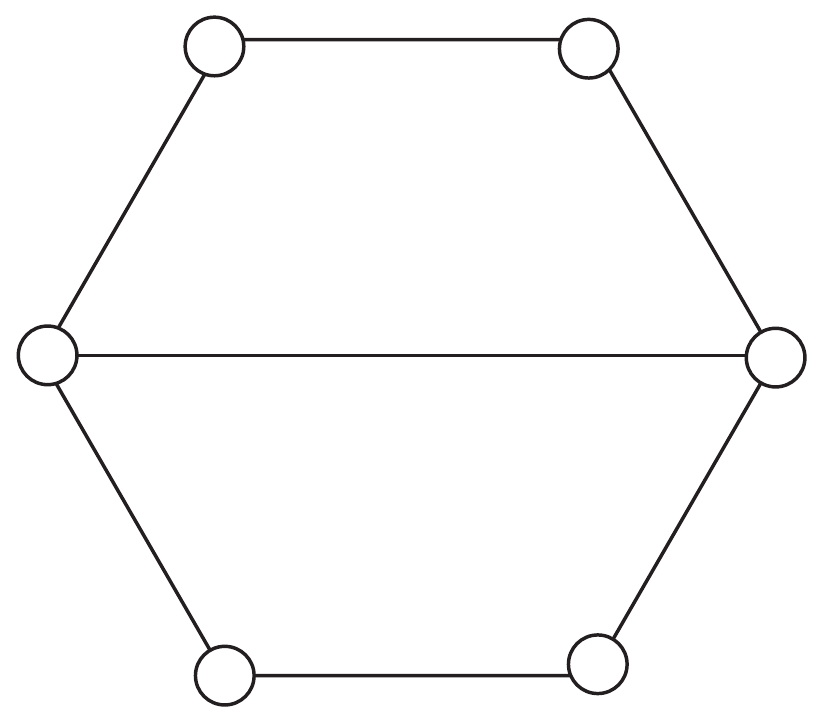}
\caption{Graph with $fix(M(G))\neq fix(B(G))$. \label{f:graph}}
\end{center}
\end{figure}

\end{ex}

If $G$ is a planar graph with (planar) dual $G^*$, then $M^*(G)=M(G^*)$, i.e., the dual of the cycle matroid $M(G)$ is the cycle matroid of the graphic dual $G^*$. Since $fix(M)=fix(M^*)$, the next result is immediate.

\begin{prop}\label{P:planar} Let $G$ be a planar graph, with dual graph $G^*$. Then $fix(M(G))=fix(M(G^*))$.
\end{prop}

Thus, for instance, if $I$ is the icosahedron and $D$ the dodecahedron, we have $fix(M(I))=fix(M(D))$. One can show both of these fixing numbers are 2. By Theorem~\ref{T:autogps} (1), the automorphism group of the cycle matroid for both of these Platonic solids is $A_5\times \mathbb{Z}_2$, so these highly symmetric matroids have large automorphism groups, but small fixing numbers.

We will need the following lemma, which is due to Matthews \cite{LRM}.

\begin{lemma}[Proposition 2.4 \cite{LRM}]
\label{th:matthews}
Let $G$ be a connected graph with more than one edge which is not a cycle. Then $B(G)$ is a
connected matroid if and only if $G$ has no vertices of degree 1. 
\end{lemma}

Our main theorem in this section is the following result relating the fixing numbers of the cycle matroid and bicircular matroid of a graph $G$.

\begin{thm}\label{T:samefix} Suppose $G$ is a 3-connected graph with at least 5 vertices. Then $fix(M(G))=fix(B(G))$.
\end{thm}

To prove Theorem~\ref{T:samefix}, we will need to determine the automorphism groups of $M(G)$ and $B(G)$. 

\begin{thm}\label{T:autogps}
\begin{enumerate}
\item If $G$ is 3-connected, then $Aut(G) \cong Aut(M(G))$.
\item If $G$ is 2-connected with minimum degree 3 and at  least 5 vertices, then $Aut(G) \cong Aut(B(G))$.
\end{enumerate}
\end{thm}

Since $Aut(M(G))$ and $Aut(B(G))$ both act on the edge set of a graph, Theorem~\ref{T:samefix} is an immediate consequence of Theorem~\ref{T:autogps}. Before proving Theorem~\ref{T:autogps}, we need a lemma concerning wheels. We let $W_n$ be the wheel on $n+1$ vertices ($W_6$ is shown in Fig.~\ref{fig:wheel1}).

\begin{lemma}\label{L:wheel}
For $n \geq 4$,  $Aut(B(W_n))\cong D_n$.
\end{lemma}
\begin{proof}
Label the edges of $W_n$ as in Figure \ref{fig:wheel1}, where the rim edges are $a_1, a_2, \dots, a_n$ and the spokes are $b_1, b_2, \dots, b_n$, with $a_i$ incident to $b_i$.  Note that the three edges incident to a rim vertex form a cocircuit; these $n$ cocircuits of size 3 are the only cocircuits in $B(W_n)$ of size 3 (and there are no smaller cocircuits).  

\begin{figure}[h]\label{F:wheel}
\begin{center}
\includegraphics[width=2in]{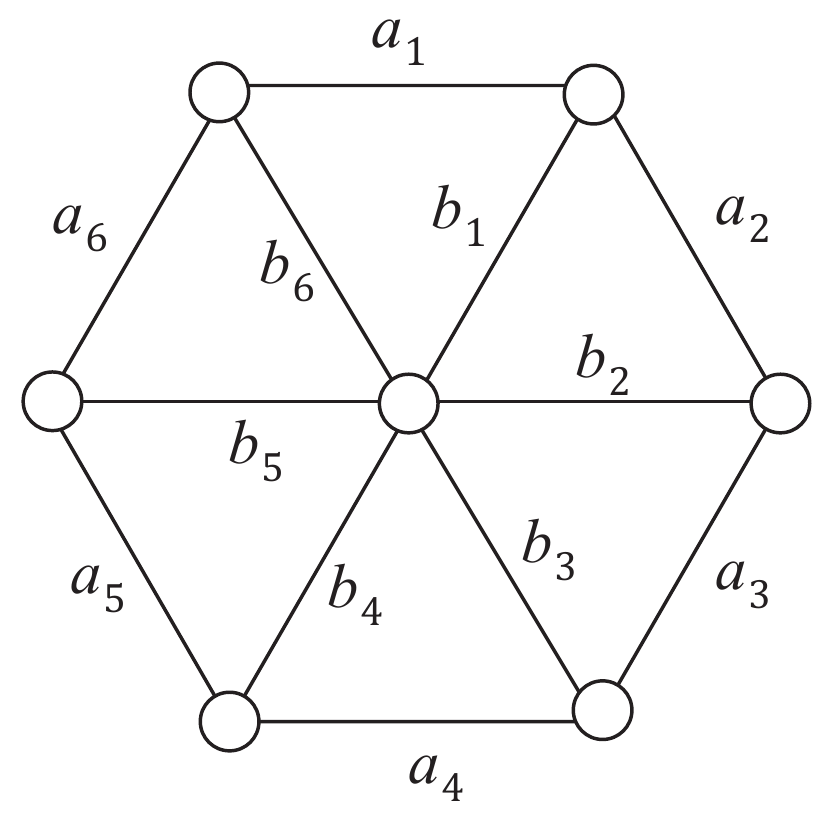}

\caption{A labeled wheel -- see the proof of Lemma~\ref{L:wheel}.}
\label{fig:wheel1}
\end{center}
\end{figure}

 Note that every dihedral symmetry of the wheel induces an automorphism of the bicircular matroid $B(W_n)$, so $D_n \leq Aut(B(W_n))$.  If $\phi \in Aut(B(W_n))$, then $\phi$ maps rim edges to rim edges (and, hence,  spoke edges to spoke edges).  This follows from the fact that each rim edge is in precisely two cocircuits of size 3 while each spoke edge is in one such cocircuit in the bicircular matroid $B(W_n)$.  Thus, the orbit of the edge $a_1$ is $orb(a_1)=\{a_1, a_2, \dots, a_n\}$ (see Figure \ref{fig:wheel1} for edge labels).

Now let $\sigma \in stab(a_1)$, the stabilizer of  $a_1$.  Then either $\sigma(a_2)=a_2$ or $\sigma(a_2)=a_n$, since $a_2$ and $a_n$ are the unique rim edges in a cocircuit of size three that contains $a_1$.  Suppose $\sigma(a_2)=a_2$.  Then we can use the cocrcuits of size 3 to show that  $\{a_1,a_2\}$ is a fixing set for $B(W_n)$. Thus, $\sigma$ is the identity.

Now assume $\sigma(a_2)=a_n$.  Then $\sigma(a_n)=a_2$ by the same cocircuit argument given above.  Letting $R_1$ denote  the automorphism corresponding to the dihedral operation of reflection through a line normal to edge $a_1$, the automorphism $\sigma R_1$ fixes $a_1$ and $a_2$, so, again, $\sigma R_1 =e$, that is, $\sigma = R_1$.

Then $stab(a_1)=\{e,R_1\}$, and so  $|Aut(B(W_n))|=2n$ by the orbit-stabilizer theorem.  Since $D_n \leq Aut(B(W_n))$, we have $Aut(B(W_n))=D_n$.
\end{proof}

We note that the proof of Lemma~\ref{L:wheel} shows that $fix(B(W_n))=2$ for $n \geq 4$. We also remark that it is straightforward to show the graph automorphism group $Aut(W_n)=D_n$.  We now prove Theorem~\ref{T:autogps}.

\begin{proof} [Proof Thm.~\ref{T:autogps}]
1.  We assume $G$ is a 3-connected graph and note that every graph automorphism induces an automorphism of the cycle matroid $M(G)$. We need to show that every matroid automorphism arises in this way. Let $v$ be a vertex of $G$ with incident edges $N(v)$. Note $N(v)$ is a cocircuit of the cycle matroid. Since $G$ is 3-connected, the induced graph $E-N(v)$ is 2-connected; hence $E-N(v)$ is a connected hyperplane in $M(G)$. Further, if $C^*$ is a cocircuit of $M(G)$ that does not correspond to a vertex star $N(v)$,  then $E-C^*$ is disconnected (as a graph), and so is a disconnected  hyperplane in $M(G)$. Thus, the connected hyperplanes of the cycle matroid have the form $E-N(v)$ for some vertex $v$.

Since matroid automorphisms preserve connected hyperplanes, we can reconstruct the vertices, and hence the graph $G$, from the cycle matroid. Thus, any matroid automorphism induces a unique graph automorphism, so $Aut(G)\cong Aut(M(G))$.

2.  Suppose $G$ is 2-connected with $n$ vertices and minimum vertex degree at least 3. As above, let $N(v)$ be the edges incident to the vertex $v$. 
Then there are two cases to consider.

{\sc Case 1:}  $G-N(v)$  is a cycle. Then  $v$ is adjacent to every other vertex of $G$; if not, $G$ would have a vertex of degree 2. Then  $G=W_{n-1}$ is a wheel  for some $n\geq 4$.  By Lemma~\ref{L:wheel}, we have $Aut(B(W_{n-1}))\cong D_{n-1} \cong Aut(W_{n-1})$.

{\sc Case 2:}   $G-N(v)$ is not a cycle. Note that $G-N(v)$ contains no vertices of degree 1 (since $G$ has minimum degree at least 3), and $G-N(v)$ is a connected graph (since $G$ is 2-connected). Thus, by Lemma~\ref{th:matthews}, the bicircular matroid on the edges of $G-N(v)$ forms a connected hyperplane in $B(G)$. Then the proof proceeds as in the cycle matroid case. 
\end{proof}

Theorem~\ref{T:samefix} now follows since 3-connected graphs have no vertices of degree 2. The next example illustrates the difference between 2-connectivity and 3-connectivity for  cycle and bicircular matroids.

\begin{ex}\label{E:2conn}
Consider the graph $G$ in Figure~\ref{F:2conn}. $G$ is 2-connected, but not 3-connected, since deleting vertices $v$ and $w$ disconnects the graph. ($G$ is a 2-sum of two copies of the graph $K_4-e$.)

\begin{figure}[h]
\begin{center}
\includegraphics[width=3in]{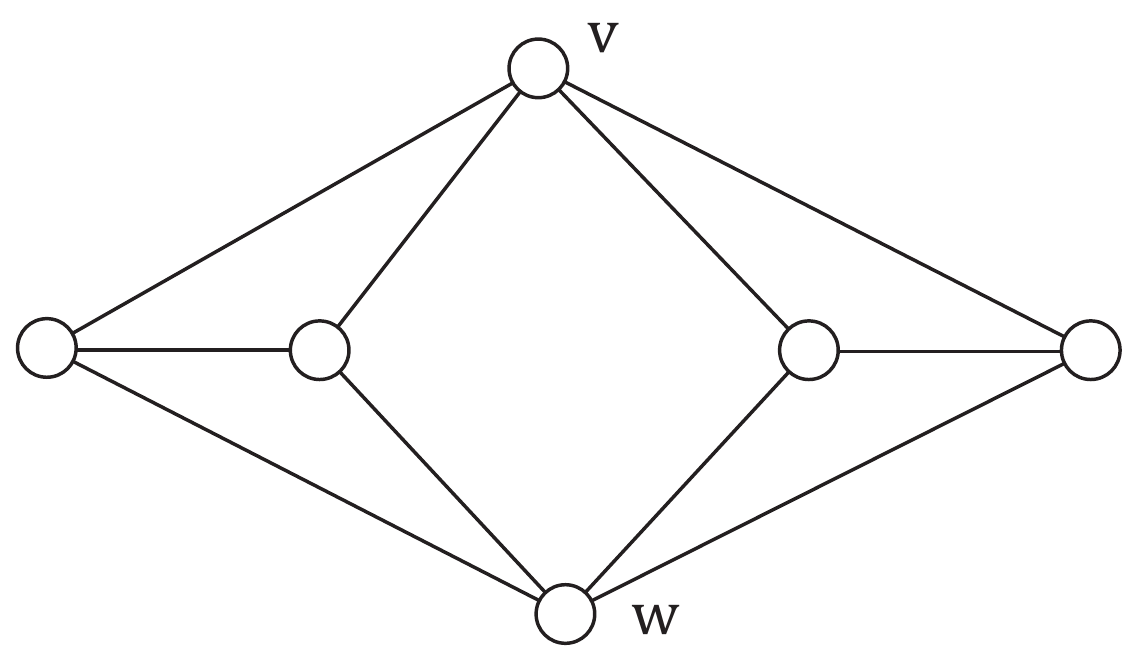}

\caption{$G$ is 2-connected, but not 3-connected.}
\label{F:2conn}
\end{center}
\end{figure}

Then $|Aut(G)|= |Aut(B(G))|=16$ and  $|Aut(M(G))|=128.$ The fact that $Aut(G)\cong Aut(B(G))$ follows from Theorem~\ref{T:autogps}(2).  For the cycle matroid, we find $Aut(M(G)) \cong (Aut(M(K_4-e))\times Aut(M(K_4-e)))\rtimes \mathbb{Z}_2$, the same symmetry group as the direct sum $M(K_4-e) \oplus M(K_4-e).$  (The semi-direct factor of $\mathbb{Z}_2$ interchanges the two copies of $K_4-e$.) 

One difference between the cycle and bicircular matroids follows from Whitney's 2-isomorphism theorem. Splitting the graph at $v$ and $w$ and reattaching (after swapping $v$ and $w$ in one component) gives an automorphism of the cycle matroid that is not present for the bicircular matroid. Although $|Aut(M(G))|>|Aut(B(G))|$, we still find  $fix(M(G))=fix(B(G))=2$, while $fix(G)=3$. 

In this example, we note that $Aut(B(G))$ is a subgroup of $Aut(M(G))$, while the reverse containment holds for Example~\ref{E:autoex1}. Thus, no subgroup relation holds for these two groups, in general. 
\end{ex} 

We conclude by applying Theorem~\ref{T:samefix} to two important classes: complete graphs and complete bipartite graphs.

\begin{cor}\label{C:complete}
\label{t:kn} Let $G=K_n$, the complete graph on $n$ vertices.
\begin{enumerate}
\item Cycle matroid:  For $n \geq 1$,  $\displaystyle{fix(M(K_n))=\left\lfloor \frac{2n}{3} \right\rfloor.}$
\item Bicircular matroid:  $fix(B(K_4))=5$.  For $n\neq 4$,  $\displaystyle{fix(B(K_n))=\left\lfloor \frac{2n}{3} \right\rfloor.}$
\end{enumerate}
\end{cor}
\begin{proof}
\begin{enumerate}
\item This follows from unpublished work of T. Maund \cite{Maund}
on the base size of the action of $S_n$ on 2-subsets of $\{1, 2, \dots, n\}$.  See Prop. 2.24 of Bailey and Cameron \cite{BailCam}.
\item  $B(K_4)\cong U_{4,6}$, the uniform matroid of rank 4 on 6 points.  Thus $fix(B(K_4))=5$.
 It is easy to check the formula  for the cases $n=1, 2$ and 3.  Now assume $n > 4$.  Then the result follows immediately from Theorem~\ref{T:samefix} and part 1.
\end{enumerate}
\end{proof}

We turn our attention to one final class of graphs, the complete bipartite graphs $K_{m,n}$.  We begin with a lemma.

\begin{lemma} \label{l:fixkmn} Let $A \subseteq E(K_{m,n})$, where $m<n$.  Let $V \cup W$ be the vertex partition with $|V|=m$ and $|W|=n$.  Suppose the edges of $A$ are incident to at least $m-1$ vertices of $V$ and $n-1$ vertices of $W$.
If $\sigma \in Aut(B(K_{m,n}))$ fixes $A$ pointwise, then $\sigma$ fixes every edge of $K_{m,n}$.
\end{lemma}
We omit the proof, which uses cocircuits in a manner similar to the proof of Lemma~\ref{L:wheel}.  Note that there is at most one vertex of $V$ and one vertex of $W$ that is not incident to an edge of $A$.

\begin{thm} Let $1\leq m<n$.  Then
\begin{enumerate}
\item $fix(M(K_{m,n}))=n-1$.
\item $fix(B(K_{m,n}))=n-1$.
\end{enumerate}

\end{thm}

\begin{proof} We prove the theorem for the bicicular case -- the cycle matroid will then follow from Theorem~\ref{T:samefix}. As in  Lemma~\ref{l:fixkmn}, let $V \cup W$ be the vertex partition in $K_{m,n}$, with $|V|=m$ and $|W|=n$, and let  $w \in W$.  Let $A$ consist of any collection of $n-1$ edges of $K_{m,n}$ incident to each vertex of $V$ and each vertex of $W$ except $w$.  (For instance, we could let $A$  be  the edges of a matching of $V$ into $W-w$ along with an arbitrary collection of $n-m-1$ edges incident to the remaining vertices of $W-w$.)

Then $|A|=n-1$.  We claim $stab(A)$ is trivial.  The proof is similar to the proof given above for $B(K_n)$.  As before, the edges incident to a vertex form a cocircuit in $B(K_{m,n})$.  The vertices in $W$ generate cocircuits of size $m$, and there are no other cocircuits in $B(K_{m,n})$ of size $m$.  Thus, if $\sigma$ is an automorphism of $B(K_{m,n})$,  $\sigma$ permutes these cocircuits.  If $\sigma$ fixes every edge of $A$,  it is straightforward to show $\sigma$ is the identity.

Finally, if $|A|<n-1$, then there are at least two vertices of $W$ not incident to any edge of $A$.  Swapping these two vertices induces a non-trivial action on the edges.  This produces a non-trivial automorphism in $stab(A)$, and so $fix(B(K_{m,n}))>n-2$.
\end{proof}

When $m=n$, we need an additional edge to obtain our fixing number.

\begin{thm} \begin{enumerate}
\item $fix(M(K_{n,n}))=n$ for $n>2$, and $fix(M(K_{2,2}))=3$.
\item $fix(B(K_{n,n}))=n$ for $n>2$, and $fix(B(K_{2,2}))=3$.
\end{enumerate}

\end{thm}
\begin{proof} The small cases are easy to check. For the cycle matroid $M(K_{m,n})$, the theorem now follows from Example 2.16 of Bailey and Cameron \cite{BailCam}.  In that example, the authors compute the base size for the wreath product $S_n \wr S_2$ as a product action.  This is the automorphism group of $M(K_{n,n})$ when $n>2$, which gives the result.

The bicircular case now follows from Theorem~\ref{T:samefix}.
\end{proof}

It would be interesting to determine the fixing number for the cycle and bicircular matroids associated with other classes of graphs.

\bibliographystyle{plain}
\bibliography{references}

\begin{thebibliography}{10}

\bibitem{AlCol}
Michael~O. Albertson and Karen~L. Collins.
\newblock Symmetry breaking in graphs.
\newblock {\em Electron. J. Combin.}, 3(1):Research Paper 18, approx.\ 17 pp.\
  (electronic), 1996.

\bibitem{BailCam}
Robert~F. Bailey and Peter~J. Cameron.
\newblock Base size, metric dimension and other invariants of groups and
  graphs.
\newblock {\em Bull. Lond. Math. Soc.}, 43(2):209--242, 2011.

\bibitem{Boutin}
Debra~L. Boutin.
\newblock Identifying graph automorphisms using determining sets.
\newblock {\em Electron. J. Combin.}, 13(1):Research Paper 78, 12 pp.
  (electronic), 2006.

\bibitem{Brualdi-White}
Richard~A. Brualdi.
\newblock Transversal matroids.
\newblock In {\em Combinatorial geometries}, volume~29 of {\em Encyclopedia
  Math. Appl.}, pages 72--97. Cambridge Univ. Press, Cambridge, 1987.

\bibitem{CGGMP}
Jos\'{e} C\'{a}ceres, Delia Garijo, Antonio Gonz\'{a}lez, Alberto M\'{a}rquez,
  and Maria~Luz Puertas.
\newblock The determining number of {K}neser graphs.
\newblock {\em Discrete Math. Theor. Comput. Sci.}, 15(1):1--14, 2013.

\bibitem{ErHar}
David Erwin and Frank Harary.
\newblock Destroying automorphisms by fixing nodes.
\newblock {\em Discrete Math.}, 306(24):3244--3252, 2006.

\bibitem{GM}
Gary Gordon and Jennifer McNulty.
\newblock {\em Matroids: A Geometric Introduction}.
\newblock Cambridge University Press, 2012.

\bibitem{LRM}
L.R. Matthews.
\newblock Bicircular matroids.
\newblock {\em Quart. J. Math. Oxford (2)}, 28:213--228, 1977.

\bibitem{Maund}
T.~Maund.
\newblock Bases for permutation groups.
\newblock {\em DPhil Thesis, University of Oxford}, 1989.

\bibitem{JO2011}
James Oxley.
\newblock {\em Matroid theory}, volume~21 of {\em Oxford Graduate Texts in
  Mathematics}.
\newblock Oxford University Press, Oxford, second edition, 2011.

\bibitem{wagon}
Stan Wagon.
\newblock Private communication.

\end{thebibliography}

\end{document}